\documentclass[10pt,twoside,reqno] {amsart}

\usepackage{amssymb}
\usepackage{pb-diagram} 
\usepackage{graphicx} 

\oddsidemargin=\evensidemargin
\catcode`\@=11
\long\def\@makefntext#1{
\protect\noindent \hbox to 3.2pt {\hskip-.9pt  
$^{{\eightrm\@thefnmark}}$\hfil}#1\hfill}		

\def\@makefnmark{\hbox to 0pt{$^{\@thefnmark}$\hss}}	
	
\def\ps@myheadings{\let\@mkboth\@gobbletwo		
\def\@oddhead{\hbox{}
\rightmark\hfil\eightrm\thepage}   
\def\@oddfoot{}\def\@evenhead{\eightrm\thepage\hfil
\leftmark\hbox{}}\def\@evenfoot{}
\def\sectionmark##1{}\def\subsectionmark##1{}}

\catcode`@=11		      		     
\def\ps@plain{\let\@mkboth\@gobbletwo
     \def\@oddhead{}\def\@oddfoot{\eightrm\hfil\thepage
     \hfil}\def\@evenhead{}\let\@evenfoot\@oddfoot}

\newcounter{sectionc}\newcounter{subsectionc}\newcounter{subsubsectionc}
\renewcommand{\section}[1] {\vspace{12pt}\addtocounter{sectionc}{1} 
\setcounter{theorem}{0} \setcounter{equation}{0}
\setcounter{subsectionc}{0}\setcounter{subsubsectionc}{0}\noindent 
	{\tenbf\thesectionc. #1}\par\vspace{5pt}}
\renewcommand{\subsection}[1] {\vspace{12pt}\addtocounter{subsectionc}{1} 
	\setcounter{subsubsectionc}{0}\noindent 
	{\bf\thesectionc.\thesubsectionc. 
	{\kern1pt \bfit #1}}\par\vspace{5pt}}
\renewcommand{\subsubsection}[1] {\vspace{12pt}
	\addtocounter{subsubsectionc}{1}
	\noindent
	{\tenrm\thesectionc.\thesubsectionc.\thesubsubsectionc.	{\kern1pt 
	\it #1}}\par\vspace{5pt}}
\newcommand{\nonumsection}[1] {\vspace{12pt}\noindent{\tenbf #1}
	\par\vspace{5pt}}

\newtheorem{theorem}{Theorem}[sectionc]

\newtheorem{proposition}[theorem]{Proposition}
\newtheorem{corollary}[theorem]{Corollary}
\newtheorem{definition}[theorem]{Definition}

\numberwithin{equation}{sectionc}

\def\abstracts#1#2#3#4{{
	\centering{\begin{minipage}{4.5in}\footnotesize\baselineskip=10pt
	\centerline{ABSTRACT} 
	\parindent=15pt #1\par 
	\parindent=15pt #2\par
	\parindent=15pt #3\par
	\parindent=15pt #4\par
	\end{minipage}}\par}} 
\def\keywords#1{{ 
	\centering{\begin{minipage}{4.5in}\footnotesize\baselineskip=10pt
	{\footnotesize\it Keywords}\/: #1
	\end{minipage}}\par}}

\newcommand{\textlineskip}{\baselineskip=13pt}
\newcommand{\smalllineskip}{\baselineskip=10pt}

\renewenvironment{thebibliography}[1]
	{\frenchspacing
	 \ninerm\baselineskip=11pt
	 \begin{list}{[\arabic{enumi}]}
	{\usecounter{enumi}\setlength{\parsep}{0pt}
	 \setlength{\leftmargin 19pt}{\rightmargin 0pt}   
	 \setlength{\itemsep}{0pt} \settowidth
	{\labelwidth}{[#1]}\sloppy}}{\end{list}}

\newcommand{\fcaption}[1]{
        \refstepcounter{figure}
        \setbox\@tempboxa = \hbox{\footnotesize Fig.~\thefigure. #1}
        \ifdim \wd\@tempboxa > 5in
           {\begin{center}
        \parbox{5in}{\footnotesize\smalllineskip Fig.~\thefigure. #1}
            \end{center}}
        \else
             {\begin{center}
             {\footnotesize Fig.~\thefigure. #1}
              \end{center}}
        \fi}
\def\runninghead#1#2{\pagestyle{myheadings}
\markboth{{\protect\footnotesize\it{\quad #1}}\hfill}
{\hfill{\protect\footnotesize\it{#2\quad}}}}

\font\tenrm=cmr10
 
\font\tenbf=cmbx10
\font\bfit=cmbxti10 at 10pt
\font\ninerm=cmr9
\font\nineit=cmti9

\font\eightrm=cmr8

\newcommand{\TryPackage}[3]{\IfFileExists{#1.sty}{\usepackage{#1}#2}{#3}}
\TryPackage{mathrsfs}{\renewcommand{\mathcal}{\mathscr}}{%
     \TryPackage{eucal}{}{}}

\newcommand{\CC}{{\mathbb C}}

\newcommand{\SLC}{{SL_2({\mathbb C})}}

\date{\today}
\begin{document}
\setlength{\textheight}{7.7truein}  

\runninghead{\quad Repeated boundary slopes for 2-bridge knots }
{Repeated boundary slopes for 2-bridge knots \quad}

\normalsize\textlineskip
\thispagestyle{empty}
\setcounter{page}{1}


\vspace*{0.88truein}

\centerline{\bf REPEATED BOUNDARY SLOPES FOR 2-BRIDGE KNOTS }
\baselineskip=13pt
\vspace*{0.37truein}

\vspace*{10pt}
\centerline{\footnotesize CYNTHIA L. CURTIS}
\baselineskip=12pt
\centerline{\footnotesize\it Department of Mathematics \& Statistics}
\baselineskip=10pt
\centerline{\footnotesize\it The College of New Jersey}
\baselineskip=10pt
\centerline{\footnotesize\it Ewing, NJ}
\baselineskip=10pt
\centerline{\footnotesize\it 08628}
\baselineskip=10pt
\centerline{\footnotesize\it {\tt ccurtis@tcnj.edu}}

\vspace{10pt}
\centerline{\footnotesize WILLIAM FRANCZAK}
\baselineskip=12pt
\centerline{\footnotesize\it Department of Mathematics}
\baselineskip=10pt
\centerline{\footnotesize\it Lehigh University}
\baselineskip=10pt
\centerline{\footnotesize\it Bethlehem, PA}
\baselineskip=10pt
\centerline{\footnotesize\it 18015}
\baselineskip=10pt
\centerline{\footnotesize\it {\tt wjf212@lehigh.edu}}

\vspace*{10pt}
\centerline{\footnotesize RANDOPLH J. LEISER}
\baselineskip=12pt
\centerline{\footnotesize\it Department of Mathematical Sciences}
\baselineskip=10pt
\centerline{\footnotesize\it New Jersey Institute of Technology}
\baselineskip=10pt
\centerline{\footnotesize\it Newark, NJ}
\baselineskip=10pt
\centerline{\footnotesize\it 07102}
\baselineskip=10pt
\centerline{\footnotesize\it {\tt rjl22@njit.edu}}

\vspace*{10pt}
\centerline{\footnotesize RYAN J. MANHEIMER}
\baselineskip=12pt
\centerline{\footnotesize\it Department of Mathematics \& Statistics}
\baselineskip=10pt
\centerline{\footnotesize\it The College of New Jersey}
\baselineskip=10pt
\centerline{\footnotesize\it Ewing, NJ}
\baselineskip=10pt
\centerline{\footnotesize\it 08628}
\baselineskip=10pt
\centerline{\footnotesize\it {\tt ryanmanheimer@gmail.com}}

\vspace*{0.225truein}

\vspace*{10pt}

\vspace*{0.21truein} 
\abstracts{We investigate the question of when distinct branched surfaces in the complement of a 2-bridge knot support essential surfaces with identical boundary slopes. We determine all instances in which this occurs and identify an infinite family of knots for which no boundary slopes are repeated.}{}{}{}

\newpage

\section{Introduction}
The essential surfaces in the complement of a 2-bridge knot were classified by Hatcher and Thurston in \cite{HT}. If $K = K(\alpha,\beta)$ is a 2-bridge knot, the surfaces are supported by branched surfaces $\Sigma[n_1,n_2,...,n_k]$, where 
$$\frac{\beta}{\alpha} = r + \frac{1}{n_1 + \frac{1}{n_2 + \frac{}{\ddots + \frac{1}{n_k}}}}$$ is a continued fraction expansion of $\beta/\alpha.$
Each such branched surface may support one, two or infinitely many connected, non-isotopic essential surfaces, all with the same boundary slope.

Boundary slopes of essential surfaces have become increasingly important computationally, for example in the computation  of Culler-Gordon-Luecke-Shalen semi-norms (see \cite{CGLS}, \cite{BZ2},\cite{C} and \cite{O}), $\SLC$ Casson invariants (see \cite{C},  \cite{BC06}, \cite{BC08}, and \cite{BC12}), and $A$-polynomials (see \cite{CCGLS}, \cite{BZ2} and \cite{BC12}). Thus, we have become interested in the question of when distinct branched surfaces support surfaces with identical boundary slopes. We call such boundary slopes {\it repeated boundary slopes.} Such slopes correspond to distinct families of ideal points in the character variety of $K(\alpha,\beta)$, and they may correspond to points on distinct curves in the character variety. Accordingly they may contribute to distinct factors of the $A$-polynomial, and their contributions will add to give the weight of the boundary slope in the Culler-Gordon-Luecke-Shalen semi-norm. To better understand this phenomenon, we began with 2-bridge knots, as here both the essential surfaces and to a lesser extent the character varieties are well-understood. (See for example \cite{BC12}, \cite{O}, and \cite{Bu}.)

In Corollary \ref{repeats} we characterize precisely when such repeats occur. From this chacterization it is clear that this will happen frequently. However in contrast in Theorem \ref{norepeats} we identify infinite families of knots for which no repeated boundary slopes occur.

We briefly outline the contents of the paper. In Section 2 we review Hatcher and Thurston's  branched surfaces $\Sigma[n_1,...,n_k]$ along with the associated essential surfaces and their boundary slopes.  In Section 3 we develop an algorithm for determining all continued fraction expansions for a given fraction, thereby identifying in a new way all of the Hatcher-Thurston branched surfaces for a given knot. Finally in Section 4 we determine the boundary slopes of the associated surfaces. We provide our own examples of repeated boundary slopes arising from certain symmetries and note the prevalence of repeated boundary slopes not arising from any clear symmetry. Finally we provide an infinite family of knots which have no repeated slopes.

\section{Incompressible surfaces for 2-bridge knots}
Throughout the paper, let $K = K(\alpha,\beta)$ be a 2-bridge knot, where $\alpha > 0$. Recall that the mirror image of $K(\alpha, \beta)$ is $K(\alpha, -\beta)$, and since two 2-bridge knots $K(\alpha,\beta)$ and $K(\alpha',\beta')$ are equivalent if and only if $\alpha = \alpha'$ and $\beta' \equiv\beta^{\pm 1}$ mod $\alpha$, we note that $K(\alpha, -\beta)$ is equivalent to $K(\alpha,\alpha-\beta)$. Thus, replacing $K(\alpha,\beta)$ with its mirror image if necessary, we assume $0<\beta<\alpha/2$. (For details on these equivalences, see \cite{BZ85}.)

We begin with a review of Hatcher and Thurston's results. 
A surface $S$ in a 3-manifold-with-boundary $M$ is said to be \textit{incompressible} if for any disc $D \subset M$ with $D \cap S = \partial D$, there exists a disc  $D' \subset S$, with $\partial D' = \partial D$.  A surface $S$ is $\partial${\em - incompressible} if for each disc $D \subset M$ with $D \cap S = \partial_+D$ and $D \cap \partial M = \partial_-D$ there is a disc $D' \subset S$ with $\partial_+ D' =\partial_+ D$ and $\partial_- D'  \subset \partial S$. A surface $S\subset M$ is \textit{essential} if it is both incompressible and $\partial$-incompressible and if no component of $S$ is boundary-parallel.

If $M$ is the complement of a tubular neighborhood $\tau(K)$ of a knot $K$ in $S^3$, where $m$ and $\ell$ are the meridian and longitude of $K$, respectively, then any essential surface $S$ in $M$ has boundary a collection of parallel curves on $\partial M$. The homology class of these curves may be described as $pm+q\ell$ for some relatively prime integers $p$ and $q$, and we call $p/q$ the \textit{boundary slope} of $S$.

Let $K(\alpha,\beta)$ be a 2-bridge knot, and let $$\frac{\beta}{\alpha} = r + \frac{1}{n_1 + \frac{1}{n_2 + \frac{}{\ddots + \frac{1}{n_k}}}}$$  be a continued fraction expansion of $\beta/\alpha.$ Henceforth we denote such a continued fraction expansion by $[n_1,...,n_k]$. 

The knot $K$ bounds a branched surface consisting of twisted bands plumbed together as shown, where at each plumbing we attach two disks: both the horizontal square depicted and its complement in the plane shown, compactified via a single point at infinity. This branched surface carries a family $S_n(t_1,...,t_{k-1})$ of surfaces, where $n \geq 1$ and $0 \leq t_i \leq n$ consisting of $n$ parallel copies of each band joined by $t_i$ copies of the outer plumbing disk and $n-t_i$ copies of the inner plumbing square at the $i^{th}$ plumbing. 

\begin{figure}[hb]
\begin{center}
\leavevmode\hbox{}
\includegraphics[scale=0.33]{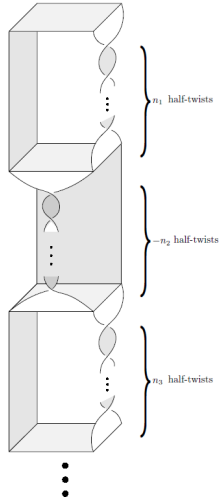}
\caption{The branched surface corresponding to $[n_1,n_2,...,n_k]$} 
\label{branchedsurface}
\end{center}
\end{figure}

Please note that our convention on the signs of the twisted bands follows that of Ohtsuki in \cite{O} and  differs from that of Hatcher and Thurston in \cite{HT}.

The following theorem is a combination of results from \cite{HT}, Theorem 1, Propostion 1, and Proposition 2:
\begin{theorem}
A surface in $S^3 - \tau(K(\alpha,\beta))$ is essential if and only if it is isotopic to one of the surfaces $S_n(t_1,...,t_{k-1})$ for a continued fraction expansion $N = [n_1,n_2,...,n_k]$ of $\beta/\alpha$ for which $|n_i| \geq 2$ for every $i$. These surfaces satisfy the following properties:
\begin{itemize}
	\item Surfaces corresponding to distinct continued fraction expansions of $\beta/\alpha$ are not isotopic. Isotopy among surfaces $S_n(t_1,t_2,...,t_{k-1})$ corresponding to a single continued fraction expansion $[n_1, n_2,...,n_k]$ is determined by the relation: $S_n(t_1,t_2,...,t_{k-1})$ is isotopic to $S_n(t_1,t_2,...,t_{i-1} + 1, t_i+1,..,t_{k-1})$ if $|n_i|=2$. If $i=1$ this means that $S_n(t_1,t_2,...,t_{k-1})$ is isotopic to $S_n(t_1+1,t_2,...,t_{k-1})$ and similarly for $i=k$.
	\item The surfaces $S_n(t_1,t_2,...,t_{k-1})$ are connected if and only if $n=1$, or $n=2$ and at least one quotient $n_i$ is odd, or $n>2$ and at least two of the quotients $n_i$ are odd.
	\item The boundary slope of the surface $S_n(t_1,t_2,...,t_{k-1})$ is given by 
	$$2[(n^+ - n^-) - (n^+_0 - n^-_0)]$$
	where $n^+ = \#{(-1)^{i+1}n_i>0}$, $n^- = \#{(-1)^{i+1}n_i<0}$, and $n^+_0$ and $n^-_0$ are the analogous counts for the unique continued fraction expansion of $\beta/\alpha$ with each $n_i$ even.

\end{itemize}
\end{theorem}

It follows from this theorem that a given branched surface may support one, several, or infinitely many connected, non-isotopic essential surfaces in the knot complement, all with the same boundary slope; indeed, the theorem completely identifies when this occurs. We are interested in the question of when a knot complement contains non-isotopic surfaces which are supported by distinct branched surfaces with identical boundary slopes.

\begin{definition}
A boundary slope of a 2-bridge knot $K(\alpha, \beta)$ is {\em repeated} if there exist distinct continued fraction expansions $[n_1,n_2,...,n_k]$ and $[m_1,m_2,...,m_l]$ such that the boundary slopes of $S_n(t_1,t_2,...,t_k)$ and $S_m(u_1,u_2,...,u_l)$ are equal.
\end{definition}

\section{Continued fraction expansions} 
In order to identify instances of repeated boundary slopes, we must better understand how to generate the collection of all continued fraction expansions $[n_1,n_2,...,n_k]$ of a given fraction $\beta/\alpha$ with $|n_i| \geq 2$ for every $i$. 

First, note that any fraction $\beta/\alpha$ may be represented by a unique continued fraction expansion $M = [m_1,m_2,...,m_k]$ with all $m_i >0$ and $m_k \geq 2$. To generate this expansion, choose each $m_i$ to be maximal such that $[m_1,m_2,...,m_i] \geq \beta/\alpha$. Note that here some $m_i$ with $i \neq k$ may equal 1, in which case $M$ will not correspond to a branched surface supporting essential surfaces in the knot complement. Nonetheless we use $M$ to generate all of the continued fraction expansions of $\beta/\alpha$ which do support essential surfaces.

To find all desired continued fraction expansions of $\beta/\alpha$ we use the following relations, where $a$ is an integer greater than 1 in the first equation and a positive integer in the second, and where $x$ is a nonzero rational number:

\begin{eqnarray}
	\frac{1}{a} & = & 1 + \frac{1}{\frac{-a}{a-1}} \nonumber \\
			& = & 1 + \frac{1}{-2 + \frac{1}{2 + \frac{1}{-2 + \frac{}							{\ddots + \frac{1}{\pm 2}}}}} \\
	\frac{1}{a + \frac{1}{x}} & = & \frac{x}{ax+1} \nonumber \\
			& = & 1 - \frac{(a-1)x + 1}{ax+1} \nonumber \\
			& = & 1 + \frac{1}{-2 + \frac{1}{2 + \frac{1}{-2 + \frac{}						{\ddots + \frac{1}{\pm (x+1)}}}}} \end{eqnarray}
Here, in each case there are $a-1$ terms in the alternating sequence $-2,2,-2,2,...,\pm 2$, and in the second formula the sign of $x+1$ is chosen to have the opposite sign of the final term in the sequence $-2,2,-2,2,...,\pm 2$, so that the entire sequence $-2,2,-2,2,...,\pm2,\mp (x+1)$ is alternating. Note that equation (3.1) is a special case of equation (3.2) if we set $x=1$ and replace $a$ in equation (3.2) with $a-1$.

We will use these relations to identify all continued fraction expansions of $\beta/\alpha$. First we require two definitions. Let $M= [m_1,m_2,...,m_k]$ be the unique continued fraction expansion of $\beta/\alpha$ such that $m_i>0$ for $1 \leq i \leq k$ and $m_k \geq 2$.  

\begin{definition} An {\em allowable sub-tuple} for $M$ is an ordered sub-tuple $(i_1,i_2, ..., i_j)$ of $(1,2,...,k)$ satisfying:
	\begin{itemize}
		\item $|i_{\ell+1} - i_{\ell}| \geq 2$ for $1 \leq \ell \leq j-1$ \\
		\item If $m_i = 1$ for any $i$ in $\{1,2,...,k\}$ then at least 				one of $i-1$, $i$, or $i+1$ is included in  						$(i_1,i_2,...,i_j)$
	\end{itemize}
\end{definition}

\begin{definition} Let $I = (i_1,i_2, ..., i_j)$ be an allowable sub-tuple for $M$. The corresponding continued fraction $M_I$ for $\beta/\alpha$ is the continued fraction expansion obtained from $M$ as follows:
	\begin{itemize}
		\item For $1 \leq \ell \leq j$ replace $m_{i_\ell}$ with an 				alternating sequence $-2,2,-2,...,\pm 2$ of length 				$m_{i_\ell} - 1$. If $m_{i_\ell} = 1$, this means that 				the term is simply deleted from $M$. \\
		\item For $1 \leq \ell \leq j$, add 1 to each of $m_{(i_\ell - 				1)}$ and $m_{(i_\ell + 1)}$. Note that if $|i_{\ell+1} - 				i_{\ell}| = 2$ this means we add 2 to $m_{(i_\ell+1)}$. 			If $i_\ell = 1$, then 	only $m_2$ is adjusted, and if 				$i_\ell = k$ then only $m_{k-1}$ is adjusted.	\\
		\item Adjust the sign of each term of the resulting 					sequence by multiplying each $m_i$ or every term of 			its replacement by $\prod_{i_\ell < i} (-1)^{m_{i_\ell}}$.  
	\end{itemize}
\end{definition}

As an example, note that if $\alpha = 73$ and $\beta = 26$, then $M = [2, 1, 4, 5]$. If $I = (2,4)$, a sub-tuple of $(1,2,3,4)$, then $M_I = [3, -6, 2, -2, 2, -2]$.

\begin{proposition}
Let $K(\alpha,\beta)$ be a 2-bridge knot in $S^3$, and let $M= [m_1,m_2,...,m_k]$ be the unique continued fraction expansion of $\beta/\alpha$ such that $m_i>0$ for $1 \leq i \leq k$.  The continued fraction expansions of $\beta/\alpha$ corresponding to branched surfaces supporting essential surfaces in $S^3 - \tau(K)$ are precisely the continued fractions $M_I$ corresponding to allowable sub-tuples $I$ for $M$.
\end{proposition}

\begin{proof}
Note that the corresponding fraction $M_I$ is obtained from $M$ by applications of (3.2) to each of the quotients $m_{i_1},m_{i_2},...,m_{i_{j-1}}$ and by applying either (3.1) if $i_j = k$ or (3.2) if $i_j<k$ to $m_{i_j}$. Thus $M_I$ is a continued fraction expansion of $\beta/\alpha$. Moreover the second condition of Definition 3.1 ensures that each quotient of $M_I = [q_1,q_2,...,q_u]$ will satisfy $|q_i| \geq 2$ for each $i$, so the corresponding branched surface will support essential surfaces in the knot complement. It remains to be seen that any continued fraction expansion of $\beta/\alpha$ which supports essential surfaces in the knot complement is of the form $M_I$ for some allowable sub-tuple $I$. 

Let $N = [n_1,n_2,...,n_\ell]$ be any continued fraction of $\beta/\alpha$ for which $|n_j| \geq 2$ for every $j$. Let $i_1>0$ be minimal such that $n_{i_1} <0$. Form a new continued fraction expansion $P_1$ from $N$ as follows:
	\begin{itemize}
		\item If $i_1 >1$, subtract 1 from $n_{(i_1-1)}$. \\
		\item Let $j_1 \geq i_1$ be minimal such that 
			either $|n_{j_1}| \neq 2$ or $\mbox{sign }n_{j_1} = 
			\mbox{sign }n_{(j_1+1)}$. Replace the 
			$(j_1- i_1)$-tuple 
			$(n_{i_1},n_{(i_1+1)},...,n_{(j_1-1)})$
			with the number $j_1 - i_1 + 1$. (Here note that if 
			$j_1 = i_1$ we simply insert a 1 before $n_{j_1}$.) \\
			\item Replace $n_{j_1}$ with 
			$|n_{j_1}|-1$. \\
			\item Multiply all terms $n_j$ with $j>j_1$ by
			$(-1)^{(j_1 - i_1 + 1)}$.
		\end{itemize}
Note that $N$ is obtained from $P_1$ by applying (3.2) (or applying (3.1) if $i_1 = \ell$) to $P_1$, letting $a$ be the $i_1^{th}$ term of $P_1$. Therefore $P_1$ is a continued fraction expansion of $\beta/\alpha$. Moreover the terms of $P_1$ replacing the first $j_1$ terms of $N$ are all positive.

Repeat the process above with $P_1$ in the role of $N$ to form a continued fraction expansion $P_2$ from $P_1$, letting $i_2>0$ be minimal such that the $i_2^{th}$ term of $P_1$ is negative. Note that $i_2 > i_1$ since the first $i_1$ terms of $P_1$ are positive by construction. 

Iterate as needed to obtain a continued fraction expansion $P_r = [p_1,p_2,...,p_s]$ of $\beta/\alpha$ with $p_i>0$ for every $i$, together with a sub-tuple $(i_1,i_2,...,i_t)$ of $(1,2,...,s)$. If $p_s \neq 1$ then $P_r$ is a continued fraction expansion of $\beta/\alpha$ with $p_i>0$ for every $i$ and with final term $\geq 2$. It follows that $M = P_r$. If $p_s = 1$, then again by the uniqueness of $M$ we have $M = [p_1,p_2,...,p_{(s-1)}+1]$. Therefore if $i_t<s$ we see by construction  that $N = M_I$, where $I = (i_1,i_2,...,i_t)$. Finally, if $i_t = s$ note that $M$ is obtained from $P_r$ by applying (3.1) with $a = p_s$. Hence $N = M_I$, where $I = (i_1,i_2,...,i_{t-1})$.

\end{proof}

Thus the branched surfaces of interest to us arise from the various continued fraction expansions $M_I$ for allowable sub-tuples $I$ for $M$. In the final section we determine the boundary slopes of the corresponding surfaces and examine the prevalence of repeated boundary slopes.

We remark that the description of continued fraction expansions in terms of allowable sub-tuples provided here may be reinterpreted geometrically in terms of the {\it state surfaces} (or {\it essential spanning surfaces}) studied in \cite{Oz}, \cite{AK}, \cite {CT},  \cite{FKP1}, \cite{FKP2}, \cite{FKP3}, and \cite{FKP4}. Specifically, the continued fraction $M$ corresponds to an alternating 2-bridge diagram of the knot, and a choice of allowable sub-tuple corresponds to a choice of chains of 1/2-twists which are to be resolved vertically in the 2-bridge diagram picture. All other chains of 1/2-twists are resolved horizontally, giving a choice of state circles for the knot. When these are joined by half-twisted bands an essential surface with boundary the knot is obtained. We will not use this geometric interpretation here, but this may be of independent interest.

\section{Repeated boundary slopes}

Now recall that the boundary slope of the surfaces supported by a given branched surface are given by 
$$2[(n^+ - n^-) - (n^+_0 - n^-_0)]$$
	where $n^+ = \#{(-1)^{i+1}n_i>0}$, $n^- = \#{(-1)^{i+1}n_i<0}$, and $n^+_0$ and $n^-_0$ are the analogous counts for the unique continued fraction expansion of $\beta/\alpha$ with each $n_i$ even.
Since $n_0^+$ and $n_0^-$ are independent of the sub-tuple $I$, the continued fraction expansions $M_I$ and $M_{I'}$ give rise to surfaces with identical boundary slopes if and only if the corresponding values $n^+ - n^-$ agree. Our next proposition gives a formula for $n^+ - n^-$ in terms of  $M$ and $I$.

\begin{theorem}\label{slope}
 Let $I = (i_1, i_2, ..., i_j)$ be an allowable sub-tuple of $(1,2,...,k)$, and let $M_I$ be the associated continued fraction expansion of $\beta/\alpha$. Then
$$n^+ - n^- = \delta + \sum_{\ell = 1}^j (-1)^{i_\ell} m_{i_\ell},$$
where $\delta = 0$ if $k$ is even and $\delta = 1$ if $k$ is odd.
\end{theorem}

\begin{proof} 
Partition $M$ into pairs $(m_i,m_{i+1})$ for each $i$ odd, $i<k$, leaving $m_k$ as a singleton if $k$ is odd. We determine the contribution of each pair to $n^+ - n^-$ for $M_I$. 

Since for any allowable sub-tuple we have  $|i_{\ell+1} - i_{\ell}| \geq 2$ for $1 \leq \ell \leq j-1$ at most one of $i$ and $i+1$ is in $I$. If neither $i$ nor $i+1$ is in $I$ then the terms in $M_I$ corresponding to $m_i$ and $m_{i+1}$ in $M$ have the same signs. Then one contributes to $n^+$ and the other to $n^-$, and the pair contributes 0 to $n^+ - n^-$.

If $i \in I$ then the chain of terms in $M_I$ corresponding to the pair $(m_i,m_{i+1})$ is the alternating sequence $\prod_{i_\ell \in I, i_\ell< i} (-1)^{m_{i_\ell}} [-2,2,...,\mp 2, \pm y]$ of length $m_i$, where $y=m_{i+1}+1$ if $i+2 \notin I$ and $y = m_{i+1}+2$ if $i+2 \in I$. The number of terms of $M_I$ before this chain is 
$$\sum_{s<i,s \notin I} 1 + \sum_{s<i,s \in I} (m_s - 1) = \sum_{s<i} 1 + \sum_{s<i,s \in I} m_s + \sum_{s<i, s \in I} (-2).$$
In this final expression, since $i$ is odd, the first sum is even, and clearly the last sum is even. Further $\sum_{s<i,s \in I} m_s$ is even if and only if  $\prod_{i_\ell \in I, i_\ell< i} (-1)^{m_{i_\ell}}$ is 1.  Thus, the chain of terms in $M_I$ corresponding to $(m_i,m_{i+1})$ either begins with a  negative number in an odd position or begins with a positive number in an even position. It follows that the entire chain contributes to $n^-$, and the contribution of the chain to $n^+ - n^-$ is $-m_i = (-1)^i m_i$.

If $i+1 \in I$ then the chain of terms in $M_I$ corresponding to the pair $(m_i,m_{i+1})$ is the alternating sequence $\prod_{i_\ell \in I, i_\ell < i} (-1)^{m_{i_\ell}} [z, -2,2,...,\mp 2]$ of length $m_{i+1}$, where $z = m_i+1$ if $i-1 \notin I$ and $z=m_i+2$ if $i-1 \in I$. As above, the sign $\prod_{i_\ell \in I, i_\ell< i} (-1)^{m_{i_\ell}}$ of this chain is positive if the number of terms in $M_I$ before the chain 
$ \sum_{s<i} 1 + \sum_{s<i,s \in I} m_s + \sum_{s<i, s \in I} (-2)$ is even.
Thus, the chain of terms in $M_I$ corresponding to $(m_i,m_{i+1})$ either begins with a  positive number in an odd position or begins with a negative number in an even position. It follows that the chain contributes to $n^+$, and the contribution of the chain to $n^+ - n^-$ is $m_{i+1} = (-1)^{i+1} m_{i+1}$.

Finally consider the contribution of $m_k$ if $k$ is odd. If $k \notin I$, this becomes a single term in $M_I$ equal to either $\prod_{i_\ell \in I}(-1)^{m_{i_\ell}} m_k$ if $k-1 \notin I$ or $\prod_{i_\ell \in I}(-1)^{m_{i_\ell}} m_k +1$ if $k-1 \in I$. As before we find this is either a positive number in an odd slot or a negative number in an even slot, so the term contributes 1 to $n^+$. On the other hand if $k \in I$ then the corresponding chain of terms in $M_I$ is $\prod_{i_\ell \in I, i_\ell <k} (-1)^{m_{i_\ell}} [-2,2,-2,... \mp 2]$ of length $m_k - 1$. Using a similar analysis as that above, we find this chain either begins with a negative number in an odd slot or begins with a positive number in an even slot. Therefore the chain contributes to $n^-$, and we see that the contribution to $n^+ - n^-$ is $-(m_k - 1) = (-1)^k m_k + 1.$ 

Adding these contributions yields $n^+ - n^- = \delta + \sum_{i_\ell \in I} (-1)^{i_\ell} m_{i_\ell}$, as desired.
\nopagebreak
\end{proof}
The following corollary is immediate, since $\delta$ depends only on $M$:
\begin{corollary}\label{repeats}
Let $K(\alpha, \beta)$ be a 2-bridge knot, and let $M = [m_1,m_2,...,m_k]$ be the unique continued fraction expansion of $\beta/\alpha$ for which $m_i > 0$ for $i =1,2,...,k$ and $m_k\geq 2$. Then $K$ has a repeated boundary slope if and only if there exist allowable sub-tuples $I= (i_1,i_2,...,i_j)$ and $J = (j_1,j_2,...,j_\ell)$ of $(1,2,...,k)$ such that 
$$\sum_{s = 1}^j (-1)^{i_s}m_{i_s} = \sum_{t = 1}^\ell (-1)^{j_t} m_{j_t}.$$
\end{corollary}

Thus, it is easy to build examples of knots with repeated boundary slopes. One collection of examples arises from symmetry. As above, let $K(\alpha,\beta)$ be a 2-bridge knot in $S^3$, and let $M= [m_1,m_2,...,m_k]$ be the unique continued fraction expansion of $\beta/\alpha$ such that $m_i>0$ for $1 \leq i \leq k$ and $m_k \geq 2$. We restrict our attention to knots $K$ for which $M$ is symmetric, so that $k$ is odd and $m_i = m_{k+1-i}$ for $1 \leq i \leq k$. In this case it is particularly easy to find examples of repeated slopes.

\begin{definition}
Let $I = (i_1,i_2,...,i_j)$ be an allowable sub-tuple for $M$. The {\em dual} $\tilde{I}$ of $I$ is the sub-tuple $(k+1-i_j,k +1- i_{j-1},...,k+1-i_1)$.
\end{definition}

\begin{proposition}
Assume $M$ satisfies the symmetry conditions that $k$ is odd and $m_i = m_{k+1-i}$ for  $1 \leq i \leq k$. Then for any allowable sub-tuple I, the boundary slopes of surfaces corresponding to $M_I$ and $M_{\tilde{I}}$ are equal.
\end{proposition}

\begin{proof} This is an immediate application of Theorem \ref{slope}. We have 

\begin{eqnarray}
	\mbox{slope for }M_I & = & 1 + \sum_{i_\ell \in I} (-1)^{i_\ell} m_{i_\ell} \\
	& = & 1 + \sum_{i_\ell \in I} (-1)^{k+1 - i_\ell} m_{k+1 - i_\ell} \\
	& = & 1 + \sum_{i_\ell \in \tilde{I}} (-1)^{i_\ell} m_{i_\ell} \\
	& = & \mbox{ slope for }M_{\tilde{I}}.
\end{eqnarray}
	
	\end{proof}

We note that if Proposition 3.3 is reinterpreted geometrically as suggested at the conclusion of Section 3, these surfaces demonstrate obvious geometric symmetry corresponding to the numeric symmetry described here.

Of course, most instances of repeated boundary slopes will not arise from the symmetry conditions in the above theorem. For example, the knot $8_{14}$ with $\alpha = 31$ and $\beta = 12$ has $M = [2,1,1,2,2]$, which is non-symmetric. Nonetheless it has two distinct continued fraction expansions corresponding to surfaces with slope 0 and two continued fraction expansions corresponding to surfaces with slope 4. The slope zero surfaces correspond to the sub-tuples $(1,3)$ and $(3,5)$, while the slope 4 surfaces arise from the sub-tuples $(3)$ and $(2,5)$.

Finally, we note that while repeated boundary slopes are common, it is still possible to find infinite families of knots with no repeated slopes. The following theorem demonstrates a construction for many such families.

\begin{theorem}\label{norepeats}
Let $M = [m_1,m_2,...,m_k]$ be a continued fraction expansion with $m_i>0$ for $1 \leq i \leq k$, $m_k \geq 2$,  such that the corresponding 2-bridge link is a knot $K$. Suppose that $m_i > \sum_{j = 1}^{i-1} m_j$ for $1 \leq i \leq k$. Then $K$ has no repeated boundary slopes.
 \end{theorem}
 
\begin{proof}
Let $I$ and $J$ be distinct allowable sub-tuples for $M$, with $I = (i_1,i_2,...,i_r)$ and $J = (j_1,j_2,...,j_s)$. By Corollary \ref{repeats} , the boundary slopes of $M_I$ and $M_J$ are equal if and only if $\sum_{\ell = 1}^r (-1)^{i_\ell} m_{i_\ell} = \sum_{\ell = 1}^{s} (-1)^{j_\ell} m_{j_\ell}$. Note too that if $i_r = j_s$, then $M_I$ and $M_J$ have the same boundary slope if and only if $M_{I'}$ and $M_{J'}$ have the same slope, where $I' = (i_1,i_2,...,i_{r-1})$ and $J' - (j_1,j_2,...,j_{s-1})$. Thus we may assume that $i_r \neq j_s$. Reversing the roles of $I$ and $J$ as needed, we assume $i_r > j_s$. 

Consider the sum $S = \sum_{\ell=1}^s (-1)^{j_\ell} m_{j_\ell} - \sum_{\ell = 1}^{r-1} (-1)^{i_\ell} m_{i_\ell}$. Note that if $i_t = j_u$ for any $t$ with $1 \leq t \leq r-1$ and $u$ with $1 \leq u \leq s$, then the $m_{i_t} = m_{j_u}$ terms will cancel in $S$, so that each term $m_\ell$ with $1 \leq \ell < i_r$ appears with total coefficient equal to -1, 0, or 1 in $S$. It follows that $|S| < \sum_{\ell = 1}^{i_{r-1}} m_\ell < m_{i_r}$. But then $(-1)^{i_r} m_{i_r} + \sum_{\ell = 1}^{r-1} (-1)^{i_\ell} m_{i_\ell} \neq \sum_{\ell=1}^s (-1)^{j_\ell} m_{j_\ell}$, and hence the boundary slopes for $M_I$ and $M_J$ are distinct.
 \end{proof}
 
Now it is easy to check, for example, that any continued fraction expansion whose entries are all odd and whose length is a multiple of 3 represents a knot. Then the following collection of continued fraction expansions gives an example of an infinite family of knots with no repeated slopes:

\begin{eqnarray*}
	 M_3 &  = & [3, 7, 15] \\
	 M_6 & =  & [3, 7, 15, 31, 63, 127] \\
		&  \cdots &  \\
	M_{3n} & = & [3, 7, 15, 31, \ldots , 3\cdot 2^{3n-1} + \small				\sum_{\ell = 0}^ {3n-2} 2^\ell ]\\
	 & \cdots &
\end{eqnarray*}

\bigskip
\noindent
{\it Acknowledgements.} The authors would like to thank Hans U. Boden for his suggestions for the paper.

\nonumsection{References}

\end{document}